\documentclass[letter]{amsart}
\usepackage{amsmath,amsfonts,amsthm,amssymb}
\usepackage{hyperref}

\theoremstyle{plain}
\newtheorem{theorem}{Theorem}[section]

\newtheorem{lemma}[theorem]{Lemma}
\newtheorem*{de-lemma}{Lemma}
\newtheorem{proposition}[theorem]{Proposition}

\theoremstyle{remark}

\theoremstyle{definition}
\newtheorem{remark}{Remark}

\newtheorem{step}{Step}

\DeclareMathOperator{\dv}{div}

\newcommand{\dd}{\mathrm{d}}
\newcommand{\R}{\mathbb{R}}

\newcommand{\C}{\mathbb{C}}

\begin{document}

\title[Gradient estimates and other related results]{Gradient estimates for semilinear elliptic systems and other related results}

\author{Panayotis Smyrnelis} \address[P.~ Smyrnelis]{Department of Mathematics\\ University of Athens\\ Panepistemiopolis\\ 15784 Athens\\ Greece}
\email[P. ~Smyrnelis]{smpanos@math.uoa.gr}
\thanks{The author was partially supported through the project PDEGE – Partial Differential Equations
Motivated by Geometric Evolution, co-financed by the European Union – European Social
Fund (ESF) and national resources, in the framework of the program Aristeia of the ‘Operational
Program Education and Lifelong Learning’ of the National Strategic Reference Framework
(NSRF)}


\begin{abstract}
A periodic connection is constructed for a double well potential defined in the plane. This solution violates Modica's estimate as well as the corresponding Liouville Theorem for general phase transition potentials. Gradient estimates are also established for several kinds of elliptic systems. They allow us to prove in some particular cases the Liouville Theorem. Finally, we give an alternative form of the stress-energy tensor for solutions defined in planar domains. As an application, we deduce a (strong) monotonicity formula.
\end{abstract}

\maketitle
\section{Introduction}
In this paper we study the possibility of extending the Modica estimate (cf. \cite{modica}) to the vector case. The Modica estimate states that for a non-negative potential $W \in C^2(\R,\R)$, and for every bounded entire solution $u \in C^3(\R^n,\R)$ of the equation
\begin{equation}\label{scalar equation}
  \Delta u = W'(u),
\end{equation}
then
\begin{equation}\label{Modica estimate}
  \frac{1}{2}|\nabla u(x)|^2 \leq W(u(x)), \ \forall x \in \R^n.
\end{equation}
A particular case occurs when $n=1$. Then, for the bounded solutions $u:\R \to \R$ of the O.D.E.
\begin{equation}\label{ODE}
  \frac{\dd^2 u}{\dd x^2}=W'(u),
\end{equation}
the \textit{Hamiltonian} $H=\frac{1}{2}|u_x|^2-W(u)$ is a non-positive constant. This law expressing the conservation of the mechanical energy follows by an integration of \eqref{ODE}.

The Modica estimate has many applications (cf. \cite{modica} and \cite{caffarelli-garofalo}). Let us mention:
\begin{itemize}
\item[1)] A Liouville type theorem: if $u:\R^n \to \R$ is a bounded solution of \eqref{scalar equation} such that $W(u(x_0))=0$ for some $x_0 \in \R^n$, then $u$ is a constant.
\item[2)] The strong monotonicity formula according to which for every bounded solution $u:\R^n \to \R$ of \eqref{scalar equation} and every $x \in \R^n$, the quotient $$\frac{1}{r^{n-1}}\int_{B(x,r)} \Big[ \frac{1}{2}|\nabla u|^2+W(u) \Big] \dd x $$ is an increasing function of $r>0$ ($B(x,r) \subset \R^n$ denotes the ball centered at $x$ of radius $r$).
\end{itemize}
Assuming that the solutions are entire is an essential hypothesis to prove the Modica estimate. We mention that other gradient bounds can be obtained for solutions of \eqref{scalar equation} defined in proper domains of $\R^n$ (cf. \cite{farina3}).  

In the vector case, for non-negative potentials $W \in C^2(\R^m,\R)$, and for bounded entire solutions $u \in C^3(\R^n;\R^m)$ of the system
\begin{equation}\label{elliptic system}
\Delta u= \nabla W(u),
\end{equation}
the Modica estimate does no longer hold. This is a well-known fact for the Ginzburg-Landau potential $W: \R^m \to \R$, $W(u)=\frac{1}{4}(|u|^2-1)^2$ (cf. \cite{farina2}, or \cite{gui}). In the present paper, we also give a counterexample which violates the Modica estimate as well as the Liouville type theorem for a double well potential defined in the plane (cf. section 2). Next, in section 3, we establish gradient estimates for several kinds of elliptic systems following the method of Caffarelli et al. (cf. \cite{caffarelli-garofalo}). Since there are no general estimates in the vector case, we show how to obtain gradient bounds in various situations. Our aim is to present a flexible technique which can easily be adapted to a more general context, or to study more specific problems. That is why, after stating several abstract theorems, we focus on the Ginzburg-Landau system \eqref{GL system}, and give in this particular case, a gradient bound which is sharp asymptotically. From these estimates, we can deduce under certain assumptions, the Liouville type theorem, and the confinement of all bounded solutions in a determined region.

In section 4, we introduce for solutions to \eqref{elliptic system} defined in planar domains, a new tool which is equivalent to the stress-energy tensor (cf. \cite{alikakosbasic} and \cite{alikakos-fal}). More precisely, we associate to every solution $u: \R^2 \supset \Omega \to \R^m$ of \eqref{elliptic system}, a function $U:\R^2 \supset \Omega  \to \R$ which solves the equation $\Delta U = 4 W(u)$. We show that the Modica estimate implies the convexity of $U$, and give as an application, a (strong) monotonicity formula for all bounded solutions $u: \R^2 \to \R$ of \eqref{scalar equation}.

\section{Construction of a periodic connection for a double well potential in the plane}
We are going to construct a double well potential: $W:\R^2 \to \R$, such that
\begin{itemize}
\item[(i)] $W(a^{\pm})=0$ with $a^{\pm}=(\pm 2,0)$, $W(u) > 0$ for $u \neq a^{\pm}$,
\item[(ii)] $D^2 W(a^{\pm})$ is a positive definite matrix,
\item[(iii)] $W$ is symmetric with respect to the coordinate axes,
\end{itemize}
and a solution $u:\R \to \R^2$ of the O.D.E. $\frac{\dd^2 u}{\dd x^2}=\nabla W(u)$ such that
\begin{itemize}
\item[(i)] $\forall x \in \R, \ u(x+T)=u(x)$ for some $T>0$ (that is, $u$ is periodic),
\item[(ii)] $u(0)=a^+$ and $u(T/2)=a^-$ ($u$ connects the minima of $W$),
\item[(iii)] the derivative of $u$ at $x=0$ or $x=T/2$ does not vanish.
\end{itemize}
Clearly, this solution violates the Modica estimate (since $\frac{1}{2} \big | \frac{\dd u}{\dd x}(0) \big |^2 > W(u(0))=0$), as well as the Liouville type theorem (since $W(u(0))=0$, $u$ is bounded and not constant). We point out that $\frac{\dd u}{\dd x}(0)$ cannot vanish, since otherwise $u$ would be constant in view of the uniqueness result for O.D.Es. To construct the solution and the potential, we proceed step by step.
\begin{step}
We consider first, a $C^{\infty}$ closed curve $\Gamma$ in the plane which is symmetric with respect to the coordinate axes, and such that
$\{(\pm 2, u_2) : u_2 \in [-1,1] \} \subset \Gamma$. $\Gamma$ will be the trajectory of our solution. We denote by $n$ the inward normal to $\Gamma$, and by $(e_1,e_2)$ the canonical basis of $\R^2$.
\end{step}
\begin{step}
In a neighborhood of $a^{\pm}$, we define $W$ as follows:
$$W(u)=2 \lambda \rho(|u-a^{\pm}|^2), \text{ for $|u_1 \mp 2| \leq 1$ and $|u_2| \leq 1$},$$
where $\lambda>0$ is a constant to be chosen, and $\rho: [0,\infty) \to [0,1/2]$ is a smooth increasing function such that
\begin{equation}
\rho(\alpha)=
\begin{cases}
\alpha, &\text{ for } 0 \leq \alpha \leq 1/4 \\
1/2, &\text{ for } \alpha \geq 3/4. \nonumber
\end{cases}
\end{equation}
\end{step}
\begin{step}
Next, we define $u$ to be the solution of $\frac{\dd^2 u}{\dd x^2}=\nabla W(u)$ with initial data $u(0)=a^+$, and $\frac{\dd u}{\dd x}(0)=e_2 /2$. Since the potential is radial, we easily see that $\nabla W(2,u_2)=(0,4\lambda\rho'(u_2^2)u_2)$, and that $u=(2,u_2)$ with
$$\frac{\dd^2 u_2}{\dd x^2}(x)=4\lambda\rho'(u_2^2)u_2, \ \frac{\dd u_2}{\dd x}(0)=\frac{1}{2}, \ u_2(0)=0.$$
In addition, we note that $u_2(x)>0$ for $x>0$. Indeed, if $u_2(t_0)=0$ for some $t_0>0$ such that $u(x)>0$ in the interval $(0,t_0)$, then we would have $u_2(0)=u_2(t_0)=0$, with $u_2$ convex and positive in $(0,t_0)$, which is a contradiction. As a consequence, $u_2$ and $\frac{\dd u_2}{\dd x}$ are increasing for $x>0$. Let $0<t_1<t_2$ be the times when $u_2(t_1)=\sqrt{3/4}$ and $u_2(t_2)=1$. Now, we choose the constant $\lambda$ such that $\frac{\dd u}{\dd x}(t_1)=e_2$. Since the Hamiltonian $H=\frac{1}{2}|u_x|^2-W(u)$ is constant along solutions, we take $\lambda$ such that $H=\frac{1}{2}\big(  \frac{1}{2} \big)^2=\frac{1}{2}-\lambda$. With this choice of $\lambda$, we still have $\frac{\dd u}{\dd x}(x)=e_2$ for $x \in [t_1,t_2]$, since by assumption, $W$ is constant on this portion of the curve.

\end{step}
\begin{step}
To extend $u$ for $x \geq t_2$, we parametrize by arc length the part of $\Gamma$ starting at the point $a^++e_2$ and ending at the point $a^-+e_2$. Let $\gamma :[t_2,t_3] \to \R^2$ be such a parametrization. Then, we set $u(x):=\gamma(x)$ for $x \in [t_2,t_3]$. Clearly, $u:[0,t_3] \to \R^2$ is smooth, since in the interval $[t_1,t_2]$, $u$ also parametrizes $\Gamma$ by arc length. In addition, we have $\frac{\dd^2 u}{\dd x^2}(x) \bot \Gamma$, for $x \in [t_1,t_3]$. To see this, just differentiate the equation $\big| \frac{\dd u}{\dd x}(x)\big|^2=1$ and note that $\frac{\dd u}{\dd x}$ is the tangent unit vector of $\Gamma$.
\end{step}
\begin{step}
Now, we define $W$ in a tubular neighborhood of the part of $\Gamma$ starting at the point $a^++e_2$ and ending at the point $a^-+e_2$ (cf. \cite{docarmo}). We set for $x \in [t_2,t_3]$ and $|\mu|\leq \epsilon \ll 1$: $W(u(x)+\mu n_{u(x)}):=\lambda +\mu \langle \frac{\dd^2 u}{\dd x^2}(x),n_{u(x)} \rangle$, where we have denoted by $\langle \cdot,\cdot \rangle$ the Euclidean inner product. By construction, for $x \in [t_2,t_3]$: 
\begin{equation}\label{nabla W sec2}
\nabla W(u(x))=\left\langle \frac{\dd^2 u}{\dd x^2}(x),n_{u(x)} \right \rangle n_{u(x)},
\end{equation} 
and $W$ is smooth in a neighborhood of the part of $\Gamma$ between the points $a^+$ and $a^-+e_2$. Indeed, at the junction of the square $\{ (u_1,u_2): \ |u_1 - 2| \leq 1, \ |u_2| \leq 1\}$ and of the tubular neighborhood, $W(u) \equiv \lambda$. Thanks to \eqref{nabla W sec2}, we also see that $u$ satisfies the equation $\frac{\dd^2 u}{\dd x^2}(x)=\nabla W(u(x))$ for $x \in [0,t_3]$. Clearly, if $\epsilon$ is small enough, we can ensure that for $x \in [t_2,t_3]$ and $|\mu|\leq \epsilon \ll 1$: $W(u(x)+\mu n)\geq \lambda/2$.
\end{step}
\begin{step}
To extend $u$ for $x \geq t_3$, we set $$u(x):=(-2,u_2(t_2+t_3-x)) \text{ for } x \in [t_3,t_2+t_3],$$ and check as in Step 3, that it is a solution of $\frac{\dd^2 u}{\dd x^2}=\nabla W(u)$. Since in the interval $[t_3,t_2+t_3-t_1]$, $u$ parametrizes $\Gamma$ by arc length, this extension is smooth at $x=t_3$. Furthermore, at time $T/2:=t_2+t_3$, we have $u(T/2)=a^-$. Next, we extend $W$ by symmetry for $u_2<0$, in a neighborhood of the remaining portion of $\Gamma$, setting $W(u_1,u_2)=W(u_1,-u_2)$. Since $W$ is also by construction symmetric with respect to the $u_2$ coordinate axis, that is, $W(u_1,u_2)=W(-u_1,u_2)$, we have $\nabla W(u)=-\nabla W(-u)$. Thus, setting $u(x):=-u(x-T/2)$ for $x \in [T/2,T]$, we define a solution of $\frac{\dd^2 u}{\dd x^2}=\nabla W(u)$ on the whole period $[0,T]$. To complete the construction, we extend $u$ periodically for all $x \in \R$, and $W$ on the whole plane in such a way that $W(u) > 0$ if $u \neq a^{\pm}$.
\end{step}

\begin{remark}\label{remark1}
Let $W:\R^2 \to \R$ be a non-negative potential satisfying for every $u \in \R^2$ such that $|u|=R>0$:
\begin{equation}\label{condition W}
W(u)=\lambda \text{ and } \nabla W(u)=-\mu u, \text{ with $\lambda, \mu>0$, two constants.}
\end{equation}
Then, we check that $u:\R \to \R^2$, $u(x)=R e^{i \sqrt{\mu} x}$ is a solution of the O.D.E. $\frac{\dd^2 u}{\dd x^2}=\nabla W(u)$, and that $H=\frac{1}{2}|u_x|^2-W(u)=\frac{R^2 \mu}{2}-\lambda$ may become positive and arbitrarily big. This situation occurs in the case of the Ginzburg-Landau potential $W(u)=\frac{1}{4}(|u|^2-1)^2$: for every $R$, $0<R<1$, we have a periodic solution of the O.D.E. called $u_R$ for which the corresponding parameters are $\lambda_R=\frac{1}{4}(R^2 -1)^2$ and $\mu_R=1-R^2$. The constant $H_R=\frac{-3R^4+4R^2-1}{4}$ is positive if and only if $\sqrt{1/3}<R<1$. Note that condition \eqref{condition W} may also be satisfied by multiple well potentials.
\end{remark}

\begin{remark}
The Modica estimate does not allow the existence of a periodic connection $u:\R \to \R$ for the scalar problem \eqref{ODE}. Indeed, if $W(u(x_0))=0$ for some $x_0 \in \R$, and $u$ is bounded, then $u_x(x_0)=0$, and by the uniqueness result for O.D.Es $u$ coincides with the constant solution $v \equiv u(x_0)$. However, for a double well potential with non-degenerate zeros $a^-$ and $a^+$, there exists a solution $u:\R \to \R$ of \eqref{ODE} (the heteroclinic connection) such that $\lim_{x \to \pm \infty}u(x)=a^{\pm}$ (cf. \cite{alikakos-fusco1} for the extension of this result to the vector case). In addition, this solution satisfies the equipartition relation $\frac{1}{2}|u_x|^2=W(u)$, that is $H=0$. Note that in the case of our counterexample, there also exists an heteroclinic connection which takes its values (by symmetry) onto the line segment $(-2,2)$ of the $u_1$ coordinate axis.
\end{remark}

\section{Gradient estimates and applications}
The proof of the Modica estimate (cf. \cite{modica}) is based on the use of the so-called $P$-functions (cf. \cite{sperb}). Let us explain in two words how they are chosen an utilized. To every solution $u:\R^n \to \R$ of the scalar equation \eqref{scalar equation}, is associated the $P$-function $P(u;x):= \frac{1}{2} |\nabla u(x) |^2 -W(u(x))$. This choice is relevant, since the function $P$ satisfies the inequality:
\begin{equation}\label{Pfunction}
  |\nabla u |^2 \Delta P \geq \frac{1}{2}|\nabla P|^2+2W'(u) \nabla u \cdot \nabla P
\end{equation}
(without any additional assumptions on $W$ or $u$).
Then, the maximum principle is applied to show that $P(u;x)\leq 0$, for every bounded solution $u$ and every $x \in \R^n$.

For system \eqref{elliptic system},  inequality \eqref{Pfunction} does no longer hold. However, it is possible under appropriate assumptions to construct other $P$-functions to which the maximum principle can be applied. More precisely, we obtain inequalities of the form $\Delta P \geq hP$, and utilize the properties satisfied by the system and the solutions to ensure that $h \geq 0$.

In this section, we establish gradient estimates for several kinds of elliptic systems following the method of Caffarelli et al. (cf. \cite{caffarelli-garofalo}). We present this technique in various situations, and point out that this approach is quite flexible and can easily be adjusted to another context. We begin with a system involving a diagonal matrix $D=\mathrm{diag}(\nu_1,\ldots,\nu_m)$ (cf. \eqref{system0} in Theorem \ref{theorem0}). The expression of the $P$-function (cf. \eqref{Ptheorem0}) is interesting in this case, since it contains the coefficients of $D$. We obtain a rough estimate (cf. \eqref{gradient estimate theorem0-statement}), which is nevertheless sufficient to prove that
\begin{itemize}
\item all bounded solutions $u:\R^n \to \R^m$ of \eqref{system0} have their images in a determined region $\omega \subset \R^m$,
\item if $u(x_0) \in \partial \omega$, for some $x_0 \in \R^n$, then the solution $u$ is constant (Liouville type theorem).
\end{itemize}
Next, in Theorem \ref{theorem1}, we consider the standard system \eqref{elliptic system} and establish a similar result, under an appropriate monotonicity assumption on the potential. Since the estimates given by the two previous Theorems are general and rough, we found it necessary to improve them by studying an important particular case. In Theorem \ref{theorem2}, we focus on the Ginzburg-Landau system \eqref{GL system}, and obtain an estimate which is sharp asymptotically (cf. \eqref{estimate3}). Finally, we consider phase transition potentials $W$, taking advantage of their convexity near the wells. Assuming that $| \nabla u(x)|$ is small enough when $u(x)$ lies outside the convexity region of $W$, we show that the solution $u$ satisfies a stronger estimate than Modica's one (cf. Theorem \ref{theorem3}). We mention that for a double well potential $W:\R \to \R$, the periodic solutions of the O.D.E. \eqref{ODE} which are near the equilibrium in the phase plane satisfy this assumption.

\begin{theorem}\label{theorem0}
Let $D=\mathrm{diag}(\nu_1,\ldots,\nu_m)$ be a $m \times m$ diagonal matrix with $\nu_i>0$, $\forall i=1,\ldots,m$. Let $A$ be a $m \times m$ matrix such that
\begin{itemize}
\item[(i)] $\langle (D^{-1}A+AD^ {-1})u,u \rangle \geq 0$, $\forall u \in \R^m$,
\item[(ii)] $\langle Au,u \rangle \geq c|u|^2$, $\forall u \in \R^m$ and for a constant $c>0$, where $|\cdot|$ and $\langle \cdot,\cdot \rangle$ denote the Euclidean norm and inner product.
\end{itemize}
Assume that $u=(u^1,\ldots,u^m)\in C^2(\R^n ; \R^m) \cap L^{\infty}(\R^n ; \R^m)$ is an entire solution of the system
\begin{equation}\label{system0}
D \Delta u+[1-\langle Au,u\rangle]u=0, \footnote{ This system reduces to system \eqref{elliptic system} only when $(A+A^{\top})=\mu D^{-1}$, for some $\mu \in \R$.}
\end{equation}
Then,
\begin{equation}\label{gradient estimate theorem0-statement}
 \sum_{j=1}^m \frac{\nu_j}{2} |\nabla u^j(x) |^2 \leq C [1-\langle Au(x),u(x) \rangle],
\end{equation}
for a constant $C(A,D, \left\| u \right\|_{L^{\infty}(\R^n;\R^m)})>0$. In particular, $\langle Au(x),u(x) \rangle \leq 1$ for every $x \in \R^n$, and if $u$ is not constant, then  $\langle Au(x),u(x) \rangle < 1$ for every $x \in \R^n$.
\end{theorem}
\begin{proof}
Fix $M>0$ and define $$\mathcal{F}_{M}=\{ u \text{ is an entire solution of \eqref{system0}} \mid \left\| u \right\|^2_{L^{\infty}(\R^n;\R^m)}\leq M \}.$$
Let $u \in \mathcal{F}_{M}$. For $j=1,\ldots,m$ and $i=1,\ldots,n$, we have
\begin{equation}
\nu_j \Delta u^j=[\langle Au,u \rangle -1]u^j, \nonumber
\end{equation}
\begin{equation}\label{laplacian u_x_i}
\nu_j \Delta u^j_{x_i}=(\langle Au_{x_i},u \rangle+\langle Au, u_{x_i} \rangle)u^j+[\langle Au,u \rangle -1]u^j_{x_i},
\end{equation}
\begin{equation}
\text{and } \Delta \Big( \frac{\nu_j}{2} |\nabla u^j |^2 \Big)=\nu_j B_j+ \nu_j \sum_{i=1}^n  \Delta u^j_{x_i} \: u^j_{x_i} , \text{ where } B_j:=\sum_{i,k=1}^n|u^j_{x_i x_k}|^2. \nonumber
\end{equation}
Therefore, utilizing \eqref{laplacian u_x_i} we obtain
\begin{equation}
\Delta \Big( \frac{\nu_j}{2} |\nabla u^j |^2 \Big)=\nu_j B_j+ \sum_{i=1}^n (\langle A u_{x_i},u \rangle +\langle Au,u_{x_i} \rangle )  u^j \: u^j_{x_i} +[\langle Au,u \rangle -1]|\nabla u^j |^2, \nonumber
\end{equation}
and
\begin{equation}\label{inequality sum}
\Delta \Big( \sum_{j=1}^m \frac{\nu_j}{2} |\nabla u^j |^2 \Big) \geq B +[\langle Au,u \rangle -1-2amM]|\nabla u |^2,
\end{equation}
where $a:=\|A\|_{L(\R^n;\R^m)}$ and $B:=\sum_{j=1}^m \nu_j B_j$.

On the other hand, we also compute
\begin{align}\label{inequality A}
\Delta \Big(\frac{1}{2} [\langle Au,u \rangle -1] \Big)&=\frac{1}{2}(\langle A \Delta u,u\rangle +\langle Au, \Delta u\rangle)+\sum_{i=1}^n \langle Au_{x_i},u_{x_i}\rangle \nonumber \\
&\geq \frac{1}{2}[\langle Au,u \rangle -1]\langle (AD^{-1}+D^{-1}A)u,u \rangle +c |\nabla u|^2,
\end{align}
since $\Delta u =[\langle Au,u \rangle -1]D^{-1}u$ and $D$ is symmetric.

Now, let $\nu:=\max_j \{ \nu_j \}$,
and let $\lambda>0$ be such that for every $v \in \R^m$, with $|v|^2\leq M$:
\begin{equation}\label{condition lambda}
(\langle Av,v \rangle -1-2am M+\lambda c ) \geq \frac{\nu}{2} \langle (AD^{-1}+D^{-1}A)v,v \rangle.
\end{equation}
Then, for every $u \in \mathcal{F}_M$ we define
\begin{equation}\label{Ptheorem0}
  P(u;x):= \Big( \sum_{j=1}^m \frac{\nu_j}{2} |\nabla u^j(x) |^2 +\frac{\lambda}{2} [\langle Au(x),u(x) \rangle -1] \Big),
\end{equation}
and thanks to \eqref{inequality sum}, \eqref{inequality A} and \eqref{condition lambda}
the inequality
\begin{equation}\label{inequality P(u)}
\Delta P(u;x) \geq  B+\langle (AD^{-1}+D^{-1}A)u(x),u(x) \rangle P(u;x)
\end{equation}
holds in $\R^n$. The remaining of the proof proceeds as in \cite{caffarelli-garofalo}. We consider $$P_{M}:=\sup \{P(u;x) \mid u \in \mathcal{F}_M, \ x \in \R^n \}$$ and suppose by contradiction that $P_{M}> 0$. Note that for $u \in \mathcal{F}_M$, $|\nabla u|$ is uniformly bounded, since $u$ and $\Delta u$ are uniformly bounded (cf. \cite{gilbarg-trudinger}, \S 3.4 p.37), and thus, $P_M$ is finite. By definition of $P_M$, there exist two sequences $(u_k)$ in $\mathcal{F}_M$ and $(x_k)$ in $\R^n$ such that $P(u_k;x_k) \to P_M$ as $k \to \infty$. Setting $v_k(x):=u_k(x+x_k)$, one can see that the sequence $(v_k)$ belong to $\mathcal{F}_M$ (since \eqref{system0} is translation invariant), and $P(v_k;0)=P(u_k;x_k) \to P_M$ as $k \to \infty$. Thanks to the fact that the first derivatives of the solutions in $\mathcal{F}_M$ satisfy a uniform bound and are equicontinuous on bounded domains (cf. Theorem 3.1. in \cite{caffarelli-garofalo} or Corollary 6.3 p.93 in \cite{gilbarg-trudinger}), one can apply the theorem of Ascoli-Arzela and deduce via a diagonal argument the existence of a solution $v \in F_{M}$, such that $P(v;0)=P_{M}$. Applying then the maximum principle to $P(v;x)$ (cf. \eqref{inequality P(u)} and Hypothesis (i)), one can see that $P(v;x)\equiv P_M$. In addition, $B \equiv 0$ and $\langle (AD^{-1}+D^{-1}A)v(x),v(x) \rangle P_M \equiv 0$. As a consequence, $v \equiv v_0$ is constant, and since $v$ is a solution, it follows that $\langle Av_0,v_0 \rangle=1$ or $v_0=0$. Thus $P_M \leq 0$, and we have proved that for every $u \in \mathcal{F}_M$ and every $x \in \R^n$:
\begin{equation}\label{gradient estimate theorem0}
  \sum_{j=1}^m \frac{\nu_j}{2} |\nabla u^j(x) |^2 \leq \frac{\lambda}{2} [1-\langle Au(x),u(x) \rangle] \Rightarrow \langle Au(x),u(x) \rangle \leq 1.
\end{equation}
To finish the proof, suppose that $\langle Au(x_0),u(x_0) \rangle = 1$ for a solution $u \in \mathcal{F}_M$, and for some $x_0 \in \R^n$. According to what precedes, $\max_{x \in \R^n}P(u;x)=P(u;x_0)=0$. Thus, by the maximum principle, we deduce as before that $P(u;x) \equiv 0$, and $B \equiv 0$, which implies that $u$ is constant.
\end{proof}

\begin{theorem}\label{theorem1}
Let $W \in C^{2,\alpha}(\R^m,\R)$ (with $0<\alpha<1$)\footnote{ This regularity assumption on $W$ ensures that every classical solution $u$ of \eqref{elliptic system} is $C^{3,\alpha}$ smooth (cf. Theorem 6.17 p.109 in \cite{gilbarg-trudinger}). As a consequence, we can compute the second derivatives of the $P$-function defined below.} be such that for some constant $R>0$:
\begin{equation}\label{condition ball}
u \in \R^m, \ |u| > R \Rightarrow u \cdot \nabla W(u) > 0.
\end{equation}
Then, if $u \in C^2(\R^n ; \R^m) \cap L^{\infty}(\R^n ; \R^m)$ is an entire solution of \eqref{elliptic system}, we have
\begin{equation}\label{estimate theorem1 statement}
  \frac{1}{2} |\nabla u(x) |^2 \leq C (R^2-|u(x)|^2),
\end{equation}
for a constant $C(W, \left\| u \right\|_{L^{\infty}(\R^n;\R^m)})>0$. In particular, 
$|u(x)|\leq R$, $\forall x \in \R^n$, and if $u$ is not constant, then  $|u(x)|< R$, $\forall x \in \R^n$.
\end{theorem}

\begin{proof}
Following Caffarelli et al., let
$$\mathcal{F}_{M}=\{ u \text{ is an entire solution of \eqref{elliptic system}} \mid \left\| u \right\|^2_{L^{\infty}(\R^n;\R^m)}\leq M \},$$
where $M>0$ is an arbitrary constant.
There exists a constant $\mu >0$ such that
\begin{equation}\label{constant mu}
  \forall u \in \R^m \text{ with } |u|^2 \leq M, \ \forall \xi \in \R^m: \ D^2W(u)(\xi,\xi)\geq -\mu |\xi|^2.
\end{equation}
We can also check that there exists a constant $\kappa >0$ such that
\begin{equation}\label{constant kappa}
 \forall u \in \R^m \text{ with } |u| \leq R: \ u \cdot \nabla W(u) \geq \kappa (|u|^2-R^2).
\end{equation}
For every $u \in \mathcal{F}_M$ we define
$$P(u;x):= \frac{1}{2} |\nabla u(x) |^2 +\frac{\kappa + \mu}{2} (|u(x)|^2-R^2).$$
We set $B:=\left( \sum_{i,j=1}^n |u_{x_i x_j}|^2 \right)$, and compute
\begin{align}\label{max principle}
\Delta P(u;x)&=B+\sum_{i=1}^n  D^2W(u)(u_{x_i},u_{x_i})+(\kappa+\mu)( |\nabla u|^2+ u \cdot \nabla W(u)) \nonumber \\
&\geq B+\kappa |\nabla u|^2+(\kappa+\mu) u \cdot \nabla W(u)   \    \text{  (cf. \eqref{constant mu})}.  \nonumber
\end{align}
Thus
\begin{equation}
\Delta P(u;x)\geq
\begin{cases}
B \geq 0, &\text{ if } |u(x)|\geq R \ \text{  (cf. \eqref{condition ball})} \\
B+ 2 \kappa P(u;x), &\text{ if } |u(x)|< R \ \text{  (cf. \eqref{constant kappa})}, \nonumber
\end{cases}
\end{equation}
and setting
\begin{equation}
h(u;x):=
\begin{cases}
0, &\text{ if } |u(x)|\geq R \\
2 \kappa, &\text{ if } |u(x)|< R, \nonumber
\end{cases}
\end{equation}
one can see that
\begin{equation}\label{max principle theorem1}
  \Delta P(u;x)\geq B+ h(u;x) P(u;x),
\end{equation}
with $h(u; \cdot) \in L^{\infty}(\R^n,\R)$, and non-negative.
Next, we consider $$P_{M}:=\sup \{P(u;x) \mid u \in \mathcal{F}_M, \ x \in \R^n \}$$ and suppose by contradiction that $P_{M}> 0$. Proceeding as in \cite{caffarelli-garofalo} and in Theorem \ref{theorem0}, we prove the existence of a solution $v \in F_{M}$, such that $P(v;0)=P_{M}$. Thanks to \eqref{max principle theorem1}, we can apply the maximum principle to $P(v;x)$, and deduce successively that $P(v;x)\equiv P_M$, $B=0$ and $v$ is constant. Utilizing \eqref{condition ball}, one can also see that $|v| \leq R$, and thus $P_M \leq 0$. This proves that for every $u \in \mathcal{F}_M$ and every $x \in \R^n$:
\begin{equation}\label{estimate theorem1}
  \frac{1}{2} |\nabla u(x) |^2 \leq \frac{\kappa + \mu}{2} (R^2-|u(x)|^2)\Rightarrow |u(x)| \leq R.
\end{equation}
To finish the proof, suppose that $|u(x_0)|= R$ for a solution $u \in \mathcal{F}_M$, and for some $x_0 \in \R^n$. According to what precedes, $\max_{x \in \R^n}P(u;x)=P(u;x_0)=0$. Thus, by the maximum principle, we deduce as before that $P(u;x) \equiv 0$, $B \equiv 0$, and $u$ is constant.
\end{proof}

\begin{remark}
Condition \eqref{condition ball} is satisfied by the symmetric phase transition potentials $W:\R^2 \simeq \C \to \R$, $W(z)=|z^N-1|^2$, with $z \in \C$ and $N \geq 2$. Indeed,
$$z \cdot \nabla W(z)=2N \mathrm{Re}(z^N(\overline{z}^N-1))\geq 2N|z|^N(|z|^N-1).$$
Clearly, Theorem \ref{theorem1} or Theorem \ref{theorem0} (with $A$ and $D$ the identity map of $\R^m$) also apply for the Ginzburg-Landau potential $W:\R^m \to \R$, $W(u)=\frac{1}{4}(|u|^2-1)^2$. In these particular cases, the solutions $u \in C^2(\R^n ; \R^m) \cap L^{\infty}(\R^n;\R^m)$ of \eqref{elliptic system}, satisfy $|u(x)|\leq 1$, $\forall x \in \R^n$. Furthermore, if $u$ is not constant, then  $|u(x)|< 1$, $\forall x \in \R^n$, and thus, the Liouville theorem holds: if $W(u(x_0))=0$ for some $x_0 \in \R^n$, then $u$ is a constant. Note that for the Ginzburg-Landau system, there is a stronger result: it is proved in \cite{farina1} that any distributional solution without any boundedness assumption, is necessarily bounded in modulus by $1$.
\end{remark}

\begin{remark}
If we just want to prove the confinement of all bounded solutions in a determined region (without obtaining a gradient estimate), a simpler $P$-function can be chosen. Let us for instance consider the solutions  $u \in C^2(\R^n ; \R^m) \cap L^{\infty}(\R^n;\R^m)$ of the system $\Delta u =F(u)$, with $F \in C^\alpha(\R^m;\R^m)$, and let $P\in C^2(\R^m,\R)$ be a function such that
$$P(u) >0 \Rightarrow  \text{(i) }\langle \nabla P(u),F(u)\rangle>0 \text{ and (ii) } D^2P(u)(\xi,\xi)\geq 0, \ \forall \xi \in \R^m.$$
Then, $P(u(x))\leq 0$, $\forall x \in \R^n$. Indeed, reproducing the previous arguments, we construct in the corresponding class $ \mathcal{F}_M$, a solution $v \in C^2(\R^n ; \R^m) \cap L^{\infty}(\R^n;\R^m)$ of $\Delta v =F(v)$, such that  $P(v(0))=P_M:=\sup \{P(u(x)) \mid u\in \mathcal{F}_M,\ x \in \R^n \}$. Thanks to (i) and  (ii), we have 
\begin{align}
\big(\Delta P(v)\big)(x)&=\langle \nabla P(v(x)),F(v(x))\rangle+\sum_{i=1}^n  D^2P(v(x))(v_{x_i}(x),v_{x_i}(x)) \nonumber\\
&>0 \text{, if $P(v(x))>0$,} \nonumber
\end{align}
which implies that $P_M=P(v(0))\leq 0$.

As an application, we can take for $P$ the distance $d$ to a convex and compact subset $K\subset \R^m$, with $C^2$ boundary. The distance is convex outside $K$, and can be extended smoothly in the interior of $K$ in such a way that $d(u,K)\leq 0$ if and only if $u \in K$. With this choice of $P$, we deduce that if $ \langle \nabla d(u),F(u)\rangle>0$ for $u \notin K$, then $u(\R^n)\subset K$ for every solution $u \in C^2(\R^n ; \R^m) \cap L^{\infty}(\R^n;\R^m)$ of $\Delta u =F(u)$.

Considering again the multiple well potential $W:\R^2  \to \R$, $W(u)=\prod_{i=1}^N|u-a_i|^2$ ($N \geq 3$), where the points $a_1,\ldots,a_N$ define a convex and closed polygon $K \subset \R^2$, one can prove that $u(\R^n)\subset K$ for every solution $u \in C^2(\R^n ; \R^m) \cap L^{\infty}(\R^n;\R^m)$ of \eqref{elliptic system}. To see this, take $P(u)=\langle u-a_k, r \rangle$, where $r$ is the outer unit normal vector to an edge of $K$ containing the vertex $a_k$. Clearly, $P$ is convex, and we easily check that $\langle \nabla W(u),r\rangle=2\sum_{i=1}^N \Big(\langle u-a_i,r \rangle \prod_{j\neq i}|u-a_j|^2\Big)>0$, when $\langle u-a_k, r \rangle>0$.
\end{remark}
Now, we are going to improve estimate \eqref{estimate theorem1 statement} for system \eqref{elliptic system} with the Ginzburg-Landau potential $W:\R^m \to \R$, $W(u)=\frac{1}{4}(|u|^2-1)^2$:
\begin{equation}\label{GL system}
  \Delta u=(|u|^2-1)u.
\end{equation}
\begin{theorem}\label{theorem2}\footnote{This improved theorem was suggested to me by Prof. Farina in a personal
 communication.}
For every non-constant solution $u \in C^2(\R^n ; \R^m)$ of \eqref{GL system}, we have for every $x \in \R^n$, $|u(x)|<1$, and the following estimate holds:
\begin{align}\label{estimate3}
\frac{1}{2}|\nabla u(x)|^2 < \sqrt{W(u)}=\frac{1}{2}(1-|u|^2).
\end{align}
\end{theorem}
\begin{proof}
Setting $Q:\R^m \to \R$, $Q(u)=\frac{|u|^2-1}{2}$, we check that $\forall u,\xi \in \R^m$:
$$|\nabla W(u)|^2=(|u|^2-1)^2|u|^2=4(Q(u))^2(2Q(u)+1),$$
$$u \cdot \nabla W(u)=(|u|^2-1)|u|^2=2Q(u)(2Q(u)+1),$$
$$D^2W(u)(\xi,\xi)=2 \langle\xi,u \rangle^2+(|u|^2-1)|\xi|^2 \geq 2 Q(u)|\xi|^2,$$
(where $\langle\cdot,\cdot \rangle$ or $\cdot$ denotes the Euclidean inner product).
Then, we proceed as before. Since the image of every solution $u$ lies in the unit ball (cf. \cite{farina1}), we consider $$\mathcal{F}_{1}=\{ u \text{ is an entire solution of \eqref{GL system}} \mid \left\| u \right\|_{L^{\infty}(\R^n,\R^m)}\leq 1\}.$$
and define $$P(u;x)=\frac{1}{2}|\nabla u(x)|^2 +Q(u(x)).$$ Let $P_{1}:=\sup \{P(u;x) \mid u \in \mathcal{F}_{1}, \ x \in \R^n \}$, and suppose by contradiction that $P_{1}> 0$. We set $B:=\left( \sum_{i,j=1}^n |u_{x_i x_j}|^2 \right)$, and compute
\begin{align}\label{max principle2}
\Delta P(u;x)&=B+\sum_{i=1}^n D^2W(u)(u_{x_i},u_{x_i})+ |\nabla u|^2+ u \cdot \nabla W(u) \nonumber \\
&\geq B+2Q(u)|\nabla u|^2+|\nabla u|^2+(2Q(u)+1) 2Q(u) \nonumber \\
&\geq B+2(2Q(u)+1)P(u;x) \geq 2 |u|^2P(u;x).
\end{align}
Proceeding as in \cite{caffarelli-garofalo}, we then prove the existence of a solution $v \in \mathcal{F}_{1}$, such that $P(v;0)=P_{1}$. Thanks to \eqref{max principle2} we can apply the maximum principle to $P(v;x)$, and deduce successively that $P(v;x)\equiv P_{1}$, $B=0$ and $P_{1}\leq 0$. 
 Thus we have proved that for every $u \in \mathcal{F}_1$ and every $x \in \R^n$:
\begin{equation}\label{estimate3large}
 \frac{1}{2}|\nabla u(x)|^2  \leq \sqrt{W(u)}=\frac{1}{2}(1-|u|^2). \nonumber
\end{equation}
By applying again the maximum principle, one can see that this inequality is strict, except for constant solutions $u \equiv u_0$ such that $|u_0|=1$.
\end{proof}

In the particular case where $n=1$, we give an even more precise result.
\begin{theorem}\label{theorem2b}
For every non-constant solution $u \in C^2(\R ; \R^m)$ of the O.D.E.
\begin{equation}\label{ODE GL}
\frac{\dd^2 u}{\dd x^2}=(|u|^2-1)u,
\end{equation}
we have for every $x \in \R^n$, $|u(x)|<1$, and the following estimate holds:
\begin{equation}\label{estimate n=1}
\frac{1}{2} \Big |\frac{\dd u}{\dd x}(x) \Big|^2 \leq
\begin{cases}
|u|^2 \sqrt{W(u)} &\text{ for } |u|^2 \geq \frac{2}{3} \smallskip \\
W(u)+\frac{1}{12} &\text{ for } |u|^2 \leq \frac{2}{3}. 
\end{cases}  
\end{equation}
In other words, the Hamiltonian $H=\frac{1}{2}|u_x|^2-W(u)$ of $u$ is less or equal than $\frac{1}{12}$, and if $S:=\sup_\R|u(x)|^2>\frac{2}{3}$, then $H\leq \frac{1}{4}(1-S)(3S-1).$
\end{theorem}
\begin{proof}
We repeat the proof of Theorem \ref{theorem2} with another choice of the $P$-function. 
We define $$P(u;x)=\frac{1}{2}\Big|\frac{\dd u}{\dd x}(x)\Big|^2 -W(u(x))+\phi(Q(u(x))),$$ where $\phi \in C^2([-1/2,0],\R)$ is strictly increasing and convex.
Next, we compute
\begin{align}\label{computation Delta P2b}
\frac{\dd^2 P}{\dd x^2}(u;x)&=\phi''(Q(u)) \langle u_{x},u \rangle^2 +\phi'(Q(u)) (| u_x|^2+ u \cdot \nabla W(u)) \nonumber \\
&\geq 2\phi'(Q(u))\frac{|u_x|^2}{2}+2\phi'(Q(u))(2Q^2(u)+Q(u)).
\end{align}
If, in addition, the function $\phi$ satisfies $3s^2+s \geq \phi(s)$, $\forall s \in [-1/2,0]$, then we have
\begin{equation}\label{max principle2b}
  \frac{\dd^2 P}{\dd x^2}(u;x)\geq h(u;x) P(u;x), \text{ with } h(u;x):=2\phi'(Q(u(x)))> 0.
\end{equation}
We construct a sequence of functions $\phi_{\epsilon}$ as follows. First, we define for every $\epsilon>0$ an increasing function $\rho_{\epsilon} \in C^{\infty}(\R,\R)$ such that
\begin{equation}
\rho_{\epsilon}(t)=
\begin{cases}
t &\text{ for } t \geq 2 \epsilon, \smallskip \\
\epsilon &\text{ for } t \leq 0,
\end{cases}   \nonumber
\end{equation}
and $\rho_{\epsilon}(t) \geq t$, $\forall t \in \R$. Then, we set $\phi_{\epsilon}(s):=\int_0^s \rho_{\epsilon}(6t+1) \dd t$ and check that this sequence has all the aforementioned properties. We also note that as $\epsilon \to 0$, $\phi_{\epsilon}$ converges uniformly on the interval $[-1/2,0]$ to the function
\begin{equation}
\phi(s)=
\begin{cases}
3s^2+s &\text{ for } s \geq -1/6, \smallskip \\
-1/12 &\text{ for } s \leq -1/6.
\end{cases}   \nonumber
\end{equation}
Proceeding as in Theorem \ref{theorem2}, we prove that $$P_{\epsilon}(u;x):=\frac{1}{2}\Big|\frac{\dd u}{\dd x}(x)\Big|^2 -W(u(x))+\phi_\epsilon(Q(u(x))) \leq 0,$$ and letting $\epsilon \to 0$ we obtain \eqref{estimate n=1}.
\end{proof}
\begin{remark}
With the help of the periodic solutions of the O.D.E. \eqref{ODE GL} that we mentioned in Remark \ref{remark1}, we are going to check the sharpness of estimates \eqref{estimate3} and \eqref{estimate n=1}. For every $0<R<1$, $$u_R:\R \to \C \simeq \R^2\subset \R^m \ (m \geq 2), \
u_R(x)=R e^{i \sqrt{1-R^2} x},$$ is a solution of \eqref{ODE GL}, and clearly $$\Big |\frac{\dd u_R}{\dd x} \Big|^2=|u_R(x)|^2(1-|u_R(x)|^2).$$ Thus, estimate \eqref{estimate n=1} is optimal for $|u|^2 \geq 2/3$, and estimate \eqref{estimate3} is sharp asymptotically, since $\frac{1}{2}\big |\frac{\dd u_R}{\dd x} \big|^2 \sim \sqrt{W(u_R)}$, as $R \to 1$. Also note that due to the existence of an heteroclinic connection, we have the following lower bound:

\begin{equation}\label{lower bound ODE}
\sup \left \{\frac{1}{2}\Big |\frac{\dd u}{\dd x} \Big|^2 : \: \text{$u$ solution of the O.D.E. \eqref{ODE GL}} \right \}\geq
\begin{cases}
|u|^2 \sqrt{W(u)} &\text{ for } |u|^2 \geq \frac{1}{3} \smallskip \\
W(u) &\text{ for } |u|^2 \leq \frac{1}{3}. \smallskip
\end{cases}  
\end{equation}
\end{remark}

The next Theorem applies in the case of phase transition potentials with $N$ non-degenerate zeros, since in a neighborhood of each of these minima the potential is convex. Note that the Ginzburg-Landau potential $W(u)=\frac{1}{4}(|u|^2-1)^2$ that we considered before, is nowhere convex inside the unit ball.
\begin{theorem}\label{theorem3}
 Let $W \in C^{2,\alpha}(\R^m,\R)$ (with $0<\alpha<1$) be a non-negative potential which is convex in the closed set $F \subset \R^m$ (that is, $D^2W(u)(\xi,\xi) \geq 0$, $\forall u \in F$, $\forall \xi \in \R^m$). Let $u \in C^2(\R^n ; \R^m) \cap L^{\infty}(\R^n ; \R^m)$ be an entire solution of \eqref{elliptic system}. We set
$$\epsilon:=\inf_{\R^m \setminus F} W,$$ 
$$S:=\sup_{u^{-1}(\R^m \setminus F)} \| \nabla u   \|^2 .$$
Then, if $0<S<\frac{2 \epsilon}{n}$, the following estimate holds:$$\frac{n}{2}|\nabla u(x)|^2 \leq \frac{\epsilon}{S}|\nabla u(x)|^2 \leq W(u(x)),\  \forall x \in \R^n.$$ In addition, if $S=0$ or $u(\R^n) \subset F$, then $u$ is constant.
\end{theorem}
\begin{proof}
We set $\lambda:=\frac{2 \epsilon}{S}$ and assume that $\lambda>n$ and $S>0$. 
We define for every bounded solution $v:\R^n \to \R^m$ of \eqref{elliptic system}, the function
$$P(v;x):= \frac{\lambda }{2} |\nabla v(x) |^2 -W(v(x)).$$
Following Caffarelli et al., let $P_{u}:=\sup_{x \in \R^n} P(u;x)$ and suppose by contradiction that $P_u>0$.
By definition of $P_u$, there exist a sequence $(x_k)$ in $\R^n$ such that $P(u;x_k) \to P_u$ as $k \to \infty$. Setting $v_k(x):=u(x+x_k)$, one can see that
\begin{itemize}
\item[(i)] the sequence $(v_k)$ is uniformly bounded in $\R^n$,
\item[(ii)] all the $v_k$ solve \eqref{elliptic system} (since \eqref{elliptic system} is translation invariant),
\item[(iii)] $\sup_{v_k^{-1}(\R^m \setminus F)} \| \nabla v_k   \|^2 \leq S$,
\item[(iv)] $P(v_k;0)=P(u;x_k) \to P_u$ as $k \to \infty$.
\end{itemize}
Thanks to the fact that the first derivatives of the sequence $(v_k)$ satisfy a uniform bound and are equicontinuous on bounded domains (cf. Theorem 3.1. in \cite{caffarelli-garofalo}), one can apply the theorem of Ascoli-Arzela and deduce via a diagonal argument the existence of a bounded solution $v:\R^n \to \R^m$ of \eqref{elliptic system}, such that $P(v;0)=P_{u}$.
Furthermore, since $v_k \to v$ and $\nabla v_k \to \nabla v$ uniformly on compact sets, we still have
\begin{equation}\label{cond nabla v}
 \sup_{v^{-1}(\R^m \setminus F)} \| \nabla v   \|^2 \leq S, \text{ and } P_{u}=\sup_{x \in \R^n} P(v;x)=P(v;0).
\end{equation}

Now, we set $B:=\left( \sum_{i,j=1}^n |v_{x_i x_j}|^2 \right)$, $A:=\big | \sum_{i=1}^n v_{x_i x_i} \big |^2$ and compute
\begin{align}\label{max principle theorem3}
\Delta P(v;x)&= \lambda B+(\lambda -1)\sum_{i=1}^n  D^2W(v)(v_{x_i},v_{x_i})-A \nonumber \\
&\geq  (\lambda-n) B+( \lambda -1)\sum_{i=1}^n  D^2W(v)(v_{x_i},v_{x_i})   \text{  (since $nB \geq A$)}, \nonumber \\
&\geq  (\lambda-n) B \geq 0, \text{ if } v(x) \in F.
\end{align}
Utilizing \eqref{cond nabla v}, we see that if $P(v;x)=P_u$, the two situations below are impossible:
\begin{itemize}
\item[(i)] $v(x) \in \R^m \setminus F$,
\item[(ii)] $v(x) \in \partial F$, and $v(\omega) \cap (\R^m \setminus F) \neq \emptyset$ for every neighborhood $\omega \subset \R^n$ of $x$.
\end{itemize}
 Thus, there exists a neighborhood $\omega \subset \R^n$ of $x$ such that $v(\omega) \subset F$, and inequality \eqref{max principle theorem3} holds in $\omega$. Applying the maximum principle, we deduce that $P(v;\cdot) \equiv P_u$ in $\omega$, and by connectedness $P(v;\cdot) \equiv P_u$ in all $\R^n$. This implies, because of \eqref{max principle theorem3}, that $B\equiv 0$, $v$ is constant and $P_u \leq 0$. Therefore, we have proved that for every $x \in \R^n$: $\frac{\lambda }{2} |\nabla u(x) |^2 \leq W(u(x))$. In the case where $S=0$, taking $\lambda \to \infty$, we see that $u$ is constant. Finally, in the case where $u(\R^n) \subset F$, we take an arbitrary $\lambda >n$, and omit in the proof the arguments involving the set $u^{-1}(\R^m \setminus F)$.
 \end{proof}

\section{An alternative form of the stress-energy tensor in the plane}

We first recall the definition of the stress-energy tensor utilized in \cite{alikakosbasic} to establish various properties of the solutions to \eqref{elliptic system}, among them the weak monotonicity formula. To every solution $u: \R^n \supset \Omega \to \R^m$ to system \eqref{elliptic system}, is associated the stress-energy tensor $T$ which is the following $n \times n$ symmetric matrix
\begin{equation}\label{tensor}
T(u) := \frac{1}{2}
\left( \begin{array}{c}
|u_{x_1}|^2 - \displaystyle{\sum_{i\neq 1}^{n} |u_{x_i}|^2 - 2 W(u)},~~ 2 u_{x_1} \!\cdot u_{x_2},~ \cdots~,~~ 2 u_{x_1} \!\cdot u_{x_n} \\
2 u_{x_2} \!\cdot u_{x_1},~~ |u_{x_2}|^2 - \displaystyle{\sum_{i\neq 2}^{n} |u_{x_i}|^2 - 2 W(u)},~ \cdots~,~~ 2 u_{x_2} \!\cdot u_{x_n}\\
\ddots \\
2 u_{x_n} \!\cdot u_{x_1},~~ 2 u_{x_n} \!\cdot u_{x_2},~ \cdots~,~~ |u_{x_n}|^2 - \displaystyle{\sum_{i\neq n}^{n} |u_{x_i}|^2 - 2 W(u)} \end{array} \right),
\end{equation}
whose elements are invariant under rotations of the coordinate system.
Note that $T(u)$ can also be written as the sum of a scalar and a symmetric matrix: $$T(u)=-\Big(\frac{1}{2}|\nabla u|^2+W(u)\Big)I_n+ \big(u_{x_i} \cdot u_{x_j}\big)_{1\leq i,j\leq n},$$ where $I_n$ denotes the identity matrix of $\R^n$. 
Setting $T=(T_1, \dots, T_n)^{\top}$ and $\dv T = ( \dv T_1, \dots, \dv T_n)^{\top}$,
the tensor has the remarkable property that $\dv T=0$ for every solution to \eqref{elliptic system}.

In this section, we give an alternative form of the stress-energy tensor $T$ in the plane. Let $\Omega \subset \R^2$ be an open and simply connected domain of the plane. We associate to every solution $u:\R^2 \supset \Omega \to \R^m$ to \eqref{elliptic system} (where $W:\R^m \to \R$ is at least $C^1$ smooth), a function $U$, which solves the equation $\Delta U = 4W(u)$. Indeed, if $u\in C^2(\Omega,\R^m)$ is a solution to \eqref{elliptic system} in $\Omega$, the equations $\dv T_1=0$ and $\dv T_2=0$ can be interpreted as the compatibility conditions:
\begin{equation}\label{Ucompatibility}
\begin{cases}
\big [|u_{x_1}|^2 -|u_{x_2}|^2 + 2 W(u) \big ]_{x_2}=\big [ 2 u_{x_1} \cdot u_{x_2}  \big]_{x_1} \\
\big [|u_{x_2}|^2 -|u_{x_1}|^2 + 2 W(u) \big ]_{x_1}=\big [ 2 u_{x_1} \cdot u_{x_2}  \big]_{x_2},
\end{cases}
\end{equation}
which ensure the existence of a function $U\in C^3(\Omega,\R)$, defined modulo an affine function, and whose Hessian matrix is

\begin{equation}\label{HessianU}
D^2U=
\left( \begin{array}{c}
|u_{x_1}|^2 - |u_{x_2}|^2 + 2 W(u),~~ 2 u_{x_1} \!\cdot u_{x_2} \\
2 u_{x_1} \!\cdot u_{x_2},~~ |u_{x_2}|^2 - |u_{x_1}|^2 + 2 W(u) \\
\end{array} \right).
\end{equation}
We note that $D^2U \equiv 0$ if and only if $W(u) \equiv 0$, $|u_{x_1}| \equiv|u_{x_2}|$, and $u_{x_1} \!\cdot u_{x_2} \equiv 0$. In particular, when $W \equiv 0$, the Hessian matrix $D^2U$ of the function $U$ is related to the Hopf differential (cf. \cite{schoen-yau}):
$$\Phi:= \frac{1}{4} \big ( \big [ |u_{x_1}|^2-|u_{x_2}|^2 \big ] -2i \langle u_{x_1},u_{x_2} \rangle \big ) \dd z \otimes \dd z, \text{ where } z:=x_1+ix_2.$$ Both are two dimensional objects that vanish if and only if the solution $u$ is conformal.

In the next Proposition, we give a boundary condition for solutions of \eqref{elliptic system} to be conformal. It is interesting to compare this result with the corresponding ones for harmonic maps (cf. \cite{smyrnelis2}).
\begin{proposition}\label{prop10}
We assume that the potential $W \in C^1(\R^m,\R)$ is non-negative.
Let $B \subset \R^2$ be a ball of radius $R$, and let $u \in C^1(\overline{B},\R^m) \cap C^2(B,\R^m)$ be a solution of \eqref{elliptic system} satisfying on $\partial B$ the boundary condition:
\begin{equation}\label{boundary condition}
  |u_{\tau}|^2-|u_{\nu}|^2+2W(u)\leq 0,
\end{equation}
where $\nu$ is the outer unit normal vector to $\partial B$, $\tau$ the tangential one, $u_{\tau}:=\nabla u \cdot \tau$, and $u_{\nu}:=\nabla u \cdot \nu$. Then, $u$ is a harmonic map which is also conformal in $B$.
\end{proposition}
\begin{proof}
Without loss of generality, we assume that $B$ is centered at the origin. We consider the polar coordinates $(r,\theta)$ and the corresponding positively oriented orthonormal basis $(\nu=x/|x|,\tau)$.
Applying Green's formula to the function $U$ we first prove that
\begin{equation}\label{pohozaevidentitysmooth2}
\int_{B} 4W(u) \dd x=\int_{\partial B}U_{\nu} =R \int_{\partial B} \big ( |u_{\tau}|^2-|u_{\nu}|^2+2W(u) \big ) \dd \sigma(x),
\end{equation}
since $U_\nu(R,\theta)=RU_{\tau \tau}(R,\theta)-\frac{1}{R}U_{\theta\theta}(R,\theta)$ and $U_{\tau \tau}:=D^2U(x)(\tau,\tau)= |u_{\tau}|^2-|u_{\nu}|^2+2W(u)$.
Next, utilizing the boundary condition \eqref{boundary condition}, we deduce that $W(u) \equiv 0$ in $B$. Thus, $u$ is harmonic, and moreover satisfies $|u_{\tau}|^2-|u_{\nu}|^2 = 0$ on $\partial B$. To conclude we apply a result for harmonic maps established in \cite{smyrnelis2}.
\end{proof}

When the solution $u$ is defined and bounded in all $\R^2$, it is known that its first derivatives are also bounded (cf. \cite{gilbarg-trudinger}, \S 3.4 p.37). In this case, the corresponding function $U$ is a solution of the equation $\Delta U=4W(u)$ in $\R^2$, with bounded second derivatives. According to the following Proposition, $U$ is the unique function, modulo a harmonic polynomial of degree 2, satisfying these properties.

\begin{proposition}\label{uniquenessU}
Let $u\in C^2(\R^2,\R^m)$ be a bounded solution to \eqref{elliptic system} in $\R^2$. Then, every solution $V$ of the equation $\Delta V=4W(u)$ in $\R^2$, with bounded second derivatives, can be written as $$V=U+\lambda (x_1^2-x_2^2)+\mu x_1x_2+\alpha x_1+\beta x_2+\gamma,$$ for constants $\lambda$, $\mu$, $\alpha$, $\beta$ and $\gamma$.
\end{proposition}

\begin{proof}
Let $V$ be a solution of the equation $\Delta V=4W(u)$ in $\R^2$, with bounded second derivatives. We define the harmonic function $h:=V-U$ in $\R^2$. Since the second derivatives of $V$ are bounded, we deduce thanks to Liouville's theorem that the second derivatives of $h$ are constants. Thus, $h$ is a harmonic polynomial of degree 2.
\end{proof}

Now, we are going to give a geometric interpretation of the Modica estimate.

\begin{proposition}\label{convexityU}
If the potential $W:\R^m \to \R$ is non-negative, and $u\in C^2(\R^2,\R^m)$ is a solution to \eqref{elliptic system}, then the corresponding function $U$ is convex if and only if
\begin{equation}\label{convexity equation}
  \big( |u_{x_1}|^2 -|u_{x_2}|^2  \big )^2+4(u_{x_1} \cdot u_{x_2} )^2 \leq 4 (W(u))^2, \ \forall x \in \R^2.
\end{equation}
Moreover,
\begin{itemize}
\item Modica's inequality (cf. \eqref{Modica estimate}) implies the convexity of the function $U$.
\item When $m=1$, the function $U$ is convex and this property is equivalent to Modica's inequality.
\end{itemize}
\end{proposition}
\begin{proof}
The function $U$ is convex if and only if $$\det(D^2 U) \geq 0 \Leftrightarrow \big( |u_{x_1}|^2 -|u_{x_2}|^2  \big )^2+4(u_{x_1} \cdot u_{x_2} )^2 \leq 4 (W(u))^2, \ \forall x \in \R^2.$$ Modica's inequality implies the last inequality for every $m \geq 1$, and is equivalent to it when $m=1$. To see this, just check that
\begin{equation}
\begin{cases}
|\nabla u|^4\geq \big( |u_{x_1}|^2 -|u_{x_2}|^2  \big )^2+4(u_{x_1} \cdot u_{x_2} )^2, &\text{ for every $m\geq 1$,}  \\
|\nabla u|^4=\big( |u_{x_1}|^2 -|u_{x_2}|^2  \big )^2+4(u_{x_1} \cdot u_{x_2} )^2, &\text{ when $m=1$.}  \nonumber
\end{cases}
\end{equation}
\end{proof}

\begin{remark}
Unfortunately, the convexity of $U$ cannot substitute the Modica estimate when $m \geq 2$. We are going to give a counterexample showing that in general the function $U$ is not convex. We consider a bounded solution $u:\R^2 \to \R^2$ of the Ginzburg-Landau system \eqref{GL system}, mentioned in \cite{gui}, and having the following two properties:
\begin{equation}\label{behavioursolutiongui}
|u(x)|=1- \frac{d^2}{2|x|^2}+o \Big(\frac{1}{|x|^2} \Big) \text{  as } |x| \to \infty, \text{ with } d \geq 1,
\end{equation}
\begin{equation}\label{hamiltonidentitysolutiongui}
\int_{\R} \big [ |u_{x_1}|^2-|u_{x_2}|^2+ 2W(u(x))  \big ] \dd x_1=0, \ \forall x_2 \in \R.
\end{equation}
From \eqref{behavioursolutiongui} and \eqref{hamiltonidentitysolutiongui}, it follows that the inequality:
\begin{equation}\label{inequalitysolutiongui}
|u_{x_1}|^2-|u_{x_2}|^2+ 2W(u) \geq 0
\end{equation}
is not satisfied in all $\R^2$, and as a consequence $U$ is not convex. Indeed, if \eqref{inequalitysolutiongui} holds in $\R^2$, then \eqref{hamiltonidentitysolutiongui} implies that $U_{x_1x_1}=|u_{x_1}|^2-|u_{x_2}|^2+ 2W(u) \equiv 0$, and integrating we find that $U(x_1,x_2)=f(x_2)x_1+g(x_2)$, where $f,g:\R \to \R$ are two smooth functions. Since $4W(u)=\Delta U=f''(x_2)x_1+g''(x_2)$ is bounded, we deduce that $f''\equiv 0$ and $4W(u)=g''(x_2)$. Finally, from the last equation and \eqref{behavioursolutiongui}, it follows that $g''\equiv 0$, $W(u) \equiv 0$, and $|u| \equiv 1$, which contradicts \eqref{behavioursolutiongui}.

Also note that a simpler counterexample invalidating the convexity of $U$ is provided by the solutions of the Ginzburg-Landau system: $u_R: \R^2 \to \R^2 \simeq \C$, $u_R(x_1,x_2)=R e^{i \sqrt{1-R^2} x_1}$.
\end{remark}

As an application of the function $U$, we are going to prove a (strong) monotonicity formula involving only the term with the potential. 
We need first to establish the following Lemma.
\begin{lemma}\label{monotonicity of Laplacian}
Let $V \in C^2(\R^n,\R)$ be a convex function, then $$r \to \frac{1}{r^{n-1}}\int_{B(x,r)} \Delta V(x) \dd x $$ is an increasing function of $r>0$ ($B(x,r) \subset \R^n$ denotes the ball centered at $x$ of radius $r$).
\end{lemma}
\begin{proof}
Without loss of generality we suppose that $x=0$. Since every $x \neq 0$ can be written $x=\rho n$ with $\rho=|x|$ and $n=x/|x|$, we have:
\begin{equation}\label{green}
  \int_{B(0,r)}\Delta V(x) \dd x= \int_{\partial B(0,r)}\frac{\partial V}{\partial n}(x) \dd \sigma (x)= r^{n-1}\int_{\partial B(0,1)}\frac{\partial V}{\partial n}(rn) \dd \sigma (n).
\end{equation}
Utilizing the convexity of $V$, we see that: $r_1 \leq r_2 \Rightarrow \frac{\partial V}{\partial n}(r_1 n) \leq \frac{\partial V}{\partial n}(r_2 n)$, for every $n \in \R^n$ such that $|n|=1$. Thus, we deduce from \eqref{green} the desired result.
\end{proof}

\begin{theorem}\label{strong monotonicity W}
Let $W \in C^2(\R^m,\R)$ be a non-negative potential, and let $u\in C^2(\R^2,\R^m)$ be a solution to \eqref{elliptic system} satisfying \eqref{convexity equation}. Then, $r \to \frac{1}{r}\int_{B(x,r)}W(u(x)) \dd x $ is an increasing function of $r>0$ ($B(x,r) \subset \R^n$ denotes the ball centered at $x$ of radius $r$). In particular, for every bounded solution $u\in C^3(\R^2,\R)$ of \eqref{scalar equation}, the previous monotonicity formula holds.
\end{theorem}
\begin{proof}
It is a straightforward consequence of Proposition \ref{convexityU} and Lemma \ref{monotonicity of Laplacian}. For bounded solutions $u\in C^3(\R^2,\R)$ of \eqref{scalar equation}, Modica's estimate holds, and thus, the corresponding function $U$ is convex.
\end{proof}
\begin{remark}
It is remarkable that an integral property, as the monotonicity formula in Theorem \ref{strong monotonicity W}, follows from a differential inequality (cf. \eqref{convexity equation}). We point out that the monotonicity formula mentioned in the Introduction, also holds for vector solutions to \eqref{elliptic system} satisfying the Modica inequality (cf. \cite{alikakosbasic}).
\end{remark}
\begin{remark}
Let us also give another application of Lemma \ref{monotonicity of Laplacian}. If $u:\R^n \to \R^m$ is a harmonic map such that $|u|^2$ is convex, then
$r \to \frac{1}{r^{n-1}}\int_{B(x,r)}|\nabla u(x)|^2 \dd x $ is an increasing function of $r>0$ (since $\Delta |u|^2=2|\nabla u|^2$).
\end{remark}

\section*{Acknowledgments}
The author wants to thank Prof. N. D. Alikakos for his continuous guidance, and also Prof. G. Fusco, Prof. A. Farina, and Prof. I. M. Sigal for their valuable suggestions and help.

\nocite{*}
\bibliographystyle{plain}

\end{document}